\documentclass[12pt,a4paper]{amsart}
\usepackage{amssymb}
\usepackage{amsthm}
\usepackage{amsmath}
\usepackage{amsxtra}
\usepackage{latexsym}
\usepackage{mathrsfs}
\usepackage{amscd}
\usepackage[all,cmtip]{xy}
\usepackage[all]{xy}
\usepackage{enumitem}
\usepackage{xcolor}
\usepackage{comment}

\newtheorem{thm}{Theorem}[section]
\newtheorem{lem}[thm]{Lemma}

\newtheorem{prop}[thm]{Proposition}
\theoremstyle{definition}
\newtheorem{dfn}[thm]{Definition}
\newtheorem{rmk}[thm]{Remark}

\numberwithin{equation}{section}
\newcommand{\var}{\overline}

\def\C{{\mathbb C}}
\def\Q{{\mathbb Q}}
\def\R{{\mathbb R}}
\def\Z{{\mathbb Z}}
\def\P{{\mathbb P}}
\def\var{\overline}

\DeclareMathOperator{\Div}{Div}
\DeclareMathOperator{\pr}{pr}
\DeclareMathOperator{\Pic}{Pic}

\DeclareMathOperator{\cl}{cl}
\DeclareMathOperator{\id}{id}
\DeclareMathOperator{\Aut}{Aut}

\usepackage{mathtools}
\mathtoolsset{showonlyrefs=true}
\newtheorem{exam}[thm]{Example}
\numberwithin{equation}{section}
\newcommand{\plim}[1][]{\mathop{\varprojlim}\limits_{#1}}
\newenvironment{parts}[0]{%
  \begin{list}{}%
    {\setlength{\itemindent}{0pt}
     \setlength{\labelwidth}{1.5\parindent}
     \setlength{\labelsep}{.5\parindent}
     \setlength{\leftmargin}{2\parindent}
     \setlength{\itemsep}{0pt}
     }%
   }%
  {\end{list}}
\newcommand{\Part}[1]{\item[\upshape#1]}
\def\vphi{{\varphi}}
\DeclareMathOperator{\Nu}{N^1}
\DeclareMathOperator{\p}{pr}
\DeclareMathOperator{\EV}{EV}
\allowdisplaybreaks[1]
\title[preperiodicity and the arithmetic degree]
{The canonical heights for Jordan blocks of small eigenvalues,
preperiodic points, and the arithmetic degrees}
\author{Kaoru Sano}
\address{Department of Mathematics, Faculty of Science, Kyoto University, Kyoto 606-8502, Japan}
\email{ksano@math.kyoto-u.ac.jp}
\begin{document}
\maketitle
\begin{abstract}
We introduce a new canonical height function
for Jordan blocks of small eigenvalues for endomorphisms
on smooth projective varieties over a number field.
We prove that under an assumption on the eigenvalues
of the endomorphism on the group of divisors modulo numerical equivalence,
the arithmetic degree at a rational point is equal to one if and only if
it is preperiodic under the endomorphism.
\end{abstract}
\tableofcontents

\section{Introduction}\label{intro}
%%=============記号のセッティング===============
Let $X$ be a smooth projective variety
and $f\colon X\longrightarrow X$ a surjective endomorphism on $X$
both defined over $\var{\Q}$.
Here an endomorphism simply means a self-morphism.
%%==============算術次数の定義=============
Fix an ample divisor $H$ on $X$ defined over $\var{\Q}$
and take a Weil height function $h_H\colon X(\var{\Q})\longrightarrow \R$
associated with $H$;
see \cite[Theorem B.3.2]{HS} and \cite[Chapter 3 and 4]{Lan}.
For a point $P\in X(\var{\Q})$,
the {\it arithmetic degree} $\alpha_f(P)$ of $f$ at $P$ is defined by
\[\alpha_f(P) :=\lim_{n\to\infty} \max\{ h_H(f^n(P)), 1 \}^{1/n}.
\]
It is known that the arithmetic degree $\alpha_f(P)$ does not depend on the choice of $H$ and $h_H$; see Remark \ref{Remark: arithmetic degree}.

%%=================Main theorem==================
This paper is motivated by the following result of Kawaguchi and Silverman:
they proved that either $\alpha_f(P)=1$, or $\alpha_f(P)$ is equal to
the complex absolute value of an eigenvalue of the linear self-map
\[
f^\ast \colon \Nu(X)_{\Q} \longrightarrow \Nu(X)_{\Q}
\]
induced by $f$; see
\cite[Remark 23]{KS3} for details.
Here, $\Nu(X)$ denotes the group of divisors on $X$ modulo numerical equivalence,
and we put $\Nu(X)_\Q:=\Nu(X)\otimes_{\Z}\Q$.

We say that a point $P\in X(\var{\Q})$ is {\it preperiodic} under $f$
if $f^m(P)=f^n(P)$ for some $m>n\geq 1$.
It is easy to see that if $P$ is preperiodic under $f$,
we have $\alpha_f(P)=1$.
In this paper, we shall provide a sufficient condition under which the converse is true.

Let $V_H$ be the $\Q$-linear subspace
of $\Pic(X)_\Q:=\Pic(X)\otimes_\Z \Q$ spanned by
the set
\[
\{(f^n)^\ast H\ |\ n\geq 0\},\]
and $\var{V_H}$ the image of $V_H$ in $\Nu(X)_\Q$.
It is known that $V_H$ and $\var{V_H}$ are finite dimensional $\Q$-vector space;
see Lemma \ref{Lemma: finiteness}.

\begin{thm}\label{Theorem: preperiodicity}
Let $X$ be a smooth projective variety defined over $\var{\Q}$, and
$f\colon X\longrightarrow X$ a surjective endomorphism
on $X$ over $\var{\Q}$.
Assume that $1$ does not appear as the complex absolute value of
an eigenvalue of the linear self-map
\[
f^\ast \colon \var{V_H} \longrightarrow \var{V_H}.
\]
We put
\[
\mu_H(f) := \min \left\{ |\lambda |\ \middle|\ \lambda \text{ is an eigenvalue of }
f^\ast \colon \var{V_H} \longrightarrow \var{V_H} \text{ with }|\lambda|\geq 1\right\}.
\]
Then for every point $P\in X(\var{\Q})$, the following conditions are equivalent.
\begin{itemize}
\item $P$ is preperiodic under $f$.
\item $\alpha_f(P)< \mu_H(f)$.
\item $\alpha_f(P)=1$.
\end{itemize}
\end{thm}

To prove this theorem, we shall introduce the notion of canonical heights
for Jordan blocks of small eigenvalues, which are defined on the set

\[
\plim[f] X(\var{\Q}):= \left\{ \text{sequence }(P_n)_{n\in\Z}\in X(\var{\Q})^{\Z}
\middle|\ f (P_n) = P_{n+1} \text{ for all }n\in\Z \right\}.
\]
For each $m\in\Z$, we set 
\begin{align*}
\pr_m \colon \plim[f] X(\var{\Q}) &\longrightarrow X(\var{\Q})\\
(P_n)_{n\in\Z} &\mapsto P_m.
\end{align*}
Note that if $f$ is an automorphism, the set $\plim[f] X(\var{\Q})$ is identified with
the set $\{f^n(P)\ |\ n\in\Z\}$, which is the union
of the forward $f$-orbit of $P$ and the forward $f^{-1}$-orbit of $P$.

The canonical heights for Jordan blocks associated
with an eigenvalue $\lambda \in \C$ satisfying $|\lambda|>1$
were introduced by Kawaguchi and Silverman in \cite[Theorem 13]{KS3}.
We generalize their results to eigenvalues whose complex absolute values
are different from $0$ and $1$ using the set $\plim[f]X(\var{\Q})$.
%%===================canonical heights======================
\begin{thm}[{The canonical heights for Jordan blocks}]\label{Theorem: canonical heights}
Let $X$ be a smooth projective variety defined over $\var{\Q}$,
and $f\colon X \longrightarrow X$ a surjective endomorphism over $\var{\Q}$.
Let $\lambda\in \C$ be a complex number with $|\lambda|\neq 0,1$.
Let
\[D_j \in \Div(X)_\C:=\Div(X)\otimes_\Z \C \quad (0\leq j\leq l)\]
be $\C$-divisors satisfying the following linear equivalences
\[
f^\ast D_j \sim D_{j-1}+\lambda D_j \quad (0\leq j \leq l),
\]
where we set $D_{-1}:=0$.
For each $D_j$, fix a Weil height function $h_{D_j}$.
Then there are unique functions
\[\widehat{h}_{D_j}\colon \plim[f] X(\var{\Q}) \longrightarrow \C \quad (0\leq j\leq l)
\]
satisfying the normalization condition
\[
\widehat{h}_{D_j}=h_{D_j}\circ \pr_0+O(1) \]
and the functional equations
\[
\widehat{h}_{D_j}\circ f = \widehat{h}_{D_{j-1}}+\lambda \widehat{h}_{D_j}\quad (0\leq j\leq l),\]
where we set $\widehat{h}_{D_{-1}}:=0$.
\end{thm}

%%=============Main theoremは偏極の場合の一般化=============
\begin{rmk}
Kawaguchi and Silverman proved the following results in \cite{KS1}.
Assume that $f\colon X\longrightarrow X$ is a polarized surjective endomorphism
(i.e., there is an ample $\R$-divisor $H$ satisfying the numerical equivalence
$f^\ast H\equiv \delta_f H$ for some $\delta_f\in \R_{>1}$).
Then we have $\alpha_f(P)\in \{1, \delta_f\}$ for any point $P\in X(\var{\Q})$.
Furthermore, we have $\alpha_f(P)=1$ if and only if $P$ is preperiodic under $f$;
see \cite[Proposition 7]{KS1} for details.
So Theorem \ref{Theorem: preperiodicity} can be regarded as
a generalization of this fact.
\end{rmk}

%======================まとめ========================
We now briefly sketch the plan of this paper.
In Section \ref{Section: Notation and definitions}, we set some notation and definitions.
In Section \ref{Section: The canonical heights},
we introduce the notion of the canonical heights for Jordan blocks
associated with eigenvalues whose complex absolute values are different from $0$ and $1$.
It simplifies the proof of Theorem \ref{Theorem: preperiodicity}.
In Section \ref{Section: On the main theorem},
we prove Theorem \ref{Theorem: preperiodicity}.
Since Theorem \ref{Theorem: preperiodicity} is a refinement
of the results proved by Kawaguchi and Silverman in \cite{KS3},
we often refer to their paper \cite{KS3}.
In Section \ref{Section: Proof of the application to KSC},
we provide an application to the relation between
the dynamical degree and the arithmetic degree;
see Proposition \ref{Theorem: application}.
We prove a conjecture proposed by Kawaguchi and Silverman
for certain endomorphisms on smooth projective varieties.
Finally, we give some remarks on the dynamical degrees and the proof of Theorem \ref{Theorem: preperiodicity}.
%%=====================Notation and definition=====================
\section{Notation and definitions}\label{Section: Notation and definitions}
In this paper, we use following notation.
\begin{itemize}
\item $\Q$, $\R$, and $\C$ are the set of
rational numbers, real numbers, and complex numbers, respectively.
\item In this paper, we work over $\var{\Q}$.
Recall that there is a Weil height function
$h\colon \var{\Q}\longrightarrow \R$ satisfying Northcott's property.
Namely, for any positive integer $d$ and for any real number $B$,
the set
\[\bigcup_{\substack{\var{\Q}/K/\Q \\ [K:\Q]\leq d}}\{x\in K \ |\ h(x)<B\}\]
is finite; see \cite[B.3]{HS} and \cite[Chapter 3 and 4]{Lan}.
\item $X$ is a smooth projective variety defined over $\var{\Q}$.
\item $H$ is an ample divisor on $X$ defined over $\var{\Q}$.
\item $f\colon X\longrightarrow X$ is a surjective endomorphism on $X$ over $\var{\Q}$.
\item $\deg f:= [\var{\Q}(X):f^\ast \var{\Q}(X)]$ is the degree of extension of function fields.
\item The {\it forward $f$-orbit} of a point $P\in X(\var{\Q})$ is the set
$\O_f(P):=\{ f^n(P)\ |\ n \geq 0\}$.
\item $P\in X(\var{\Q})$ is said to be {\it preperiodic} (resp.\ {\it wandering})
under $f$ if the forward $f$-orbit $\O_f(P)$ is finite (resp.\ infinite).
\item $\Div(X)$ is the group of divisors on $X$ defined over $\var{\Q}$.
\item $\Pic (X)$ is the Picard group of $X$.
\item $\Nu(X)$ is the quotient group of $\Pic(X)$ by the numerical equivalence.
%\item $V_H$ is a $\Q$-linear subspace of $\Pic(X)_\Q$ spanned by
%the set $\{(f^n)^\ast H\ |\ n\geq 0\}$; see Lemma \ref{Lemma: finiteness}
%\item $\var{V_H}$ is the image of $V_H$ in $\Nu(X)_\Q$.
\item For a divisor $D\in \Div(X)$, fix a Weil height function
\[
h_D\colon X(\var{\Q})\longrightarrow \R;
\]
see \cite[Theorem B. 3.2]{HS} and \cite[Chapter 4, Theorem 5.1]{Lan}.
For a $\C$-divisor $D=\sum_{i=1}^r a_i D_i \in \Div(X)_\C\ (a_i\in \C)$,
the Weil height function $h_D$ associated with $D$ is defined by
\[
h_D:=\sum_{i=1}^r a_i h_{D_i}\colon X(\var{\Q})\longrightarrow \C.\]
\end{itemize}

%%=====================算術次数に関して===============================
\begin{dfn}
We define the {\it arithmetic degree} of $f$ at $P\in X(\var{\Q})$ by
\[
\alpha_f(P):=\lim_{n\to\infty} \max\{h_H(f^n(P)), 1\}^{1/n}.
\]
\end{dfn}
%%========================算術次数の注意==========================
\begin{rmk}\label{Remark: arithmetic degree}
Kawaguchi and Silverman proved the existence of the arithmetic degree $\alpha_f(P)$;
see \cite[Theorem 3]{KS3}.
It is also known that $\alpha_f(P)$
does not depend on the choice of $H$ and of $h_H$;
see \cite[Proposition 12]{KS2}.
\end{rmk}
%%==========================固有値集合の定義=========================
\begin{dfn}
For a linear self-map $\vphi \colon V\longrightarrow V$
on a finite dimensional vector space $V$ over a subfield $K$ of $\C$,
the set of all the eigenvalues of $\vphi$ on $V\otimes_K \C$ is denoted by $\EV(\vphi;V)$.
For a real number $B\in \R$, we define
\[
\EV(\vphi,B;V):=\left\{\lambda \in \EV(\vphi ; V)\ \middle|\ |\lambda|\geq B\right\}.
\]
\end{dfn}
%%========================magnitudeの定義===========================
%\begin{dfn}
%The {\it magnitude} $\mu_H(f)$ of $f$ associated with $H$ is defined by
%\[
%\mu_H(f) := \min \left\{|\lambda |\ \middle|\ 
%\lambda\in \EV(f^\ast,1;\overline{V_H})
%\right\}.
%\]
%\end{dfn}

%%===============全軌道の定義==============
\begin{dfn}
We say a sequence $(P_n)_{n\in\Z}\in X(\var{\Q})^{\Z}$ is an {\it $f$-orbit} if it satisfies
$f( P_n)=P_{n+1}$
for all $n\in\Z$.
An {\it $f$-orbit} of $P\in X(\var{\Q})$ is
an $f$-orbit satisfying $P_0=P$.
Let $\plim[f] X(\var{\Q})$ be the set of all $f$-orbits.
For each integer $m\in\Z$,
let
\[
\p_m\colon \plim[f] X(\var{\Q})\longrightarrow X(\var{\Q}), \quad (P_n)_{n\in\Z}\mapsto P_m
\]
be the projection to the $m$-th component.
For an $f$-orbit $(P_n)_{n\in\Z}$,
let $f((P_n)_{n\in\Z})$ be the $f$-orbit
whose $m$-th component is $P_{m+1}$ for all $m\in\Z$.
Then a self-map
\[
f\colon \plim[f] X(\var{\Q})\longrightarrow \plim[f] X(\var{\Q})
\]
is defined and we regard it as the left shift operator.
We also define the right shift operator
\[
R\colon \plim[f] X(\var{\Q}) \longrightarrow \plim[f] X(\var{\Q})
\]
such that $\p_m \circ R((P_n))=P_{m-1}$ for all $m \in \Z$.
\end{dfn}

%%=================ジョルダン標準形行列の定義==============
\begin{dfn}\label{dfn: Jordan block}
For a non-negative integer $l\geq 0$ and a complex number $\lambda\in \C$,
let
\[\Lambda:= \left(
\begin{array}{ccccc}
\lambda & & & & \\
1 &\lambda & &O & \\
 &1 &\ddots & & \\
O & &\ddots &\lambda & \\
 & & &1 &\lambda
\end{array}
\right)
\]
be the Jordan block matrix of the size $(l+1)\times (l+1)$.
\end{dfn}

%%=======================絶対値の定義========================
\begin{dfn}
The symbol $\|\cdot \|$ denotes the sup norm of a (column) vector or a matrix of complex numbers,
i.e., for vectors $v={}^t(a_0,\ldots ,a_l)\in\C^{l+1}$ and matrices $A=(a_{i,j})_{0\leq i,j\leq l}$ with complex coordinates, we set
\[
\| v \|:=\max_{0\leq i\leq l} |a_i| \quad \text{and}\quad \|A\|:=\max_{0\leq i,j\leq l}|a_{i,j}|.
\]
\end{dfn}
\begin{rmk}
For a (column) vector $v\in\C^{l+1}$ and a square matrix $A$ of size $(l+1)\times (l+1)$, we have
\[
\|Av\|\leq (l+1) \cdot \|A\|\cdot \|v\|.
\]
We frequently use this inequality.
\end{rmk}

%%%%%%%%%%=================標準高さの§============================
\section{The canonical heights for Jordan blocks}\label{Section: The canonical heights}
In this section, we shall prove Theorem \ref{Theorem: canonical heights}.
We shall introduce the canonical heights
for Jordan blocks whose complex absolute values are different from $0$ and $1$.
Our canonical heights are generalizations of the canonical heights introduced by Kawaguchi and Silverman in \cite{KS3} for eigenvalues
whose complex absolute values are greater than $1$.

%%==================行列の成分の大きさの補題=================
In the following, let $\Lambda$ be the Jordan block matrix
of the size $(l+1)\times (l+1)$ as in Definition \ref{dfn: Jordan block}.
\begin{lem}\label{Lemma: matrix}
\begin{parts}
\Part{(a)}
For all $n\geq 1$, we have $||\Lambda^n|| \leq n^l \max\{|\lambda|,1\}^n$.
\Part{(b)}
If $|\lambda|<1$, for all $n\geq l$,
we have $||\Lambda^n||\leq n^l\cdot |\lambda|^{n-l}$.
\Part{(c)}
If $|\lambda|<1$, for any vector $v\in \C^{l+1}$, we have
\[
\lim_{n\to\infty}\Lambda^n v=0.
\]
\Part{(d)}If $\lambda\neq 0$, for any nonzero vector
$v\in \C^{l+1}\setminus \{0\}$, we have
\[
\lim_{n\to\infty} ||\Lambda^{n} v||^{1/n}=|\lambda|.
\]
\Part{(e)}If $\lambda\neq 0$, for any nonzero vector
$v\in \C^{l+1}\setminus \{0\}$, we have
\[
\lim_{n\to\infty} ||\Lambda^{-n} v||^{1/n}=|\lambda|^{-1}.
\]
\end{parts}
\end{lem}
%%====================証明======================
\begin{proof}
\begin{parts}
\Part{(a)}
See \cite[Lemma 12]{KS3}.
\Part{(b)} This follows from the same arguments as (a).
\Part{(c)} The assertion follows from (b).
\Part{(d)} See \cite[Lemma 12]{KS3}.
\Part{(e)}

Since the Jordan normal form of $\Lambda^{-1}$ is
$\lambda^{-1}I+N$ with a nilpotent matrix $N$,
there is an invertible matrix $U$ such that
\[
U\Lambda^{-1} U^{-1}=\lambda^{-1}I+N.
\]
Since any two norms on $\C^{l+1}$ are equivalent to each other,
there are positive real numbers $C, C'\in \R_{>0}$ such that for all $v\in\C^{l+1}$,
the following inequalities hold
\[
C||v||\leq ||U^{-1}v||\leq C'||v||.
\]
Combining these inequalities with $(d)$, we get
\begin{align}
\lim_{n\to\infty}||\Lambda^{-n}v||^{1/n}
&= \lim_{n\to\infty} ||U^{-1}(\lambda^{-1}I+N)^nUv||^{1/n}\\
&= \lim_{n\to\infty} ||(\lambda^{-1}I+N)^n(Uv)||^{1/n}\\
&= |\lambda|^{-1}.
\end{align}
\end{parts}
\end{proof}

%%=====================公理的標準高さの定義とその性質====================
\begin{prop}[The axiomatic canonical heights for Jordan blocks]
\label{Proposition: axiomatic canonical heights}
Let $\lambda\in\C$ be a complex number satisfying $|\lambda |\neq 0,1$.
Let $S$ be a set with a self-bijection $R\colon S\longrightarrow S$,
and $h\colon S\longrightarrow \C^{l+1}$ a vector valued function satisfying
\[
||h\circ R-\Lambda^{-1} h||=O(1).
\]
Then there is a unique function
$\widehat{h}\colon S\longrightarrow\C^{l+1}$
satisfying the functional equation
\[
\widehat{h}\circ R=\Lambda^{-1} \widehat{h}
\]
and the normalization condition
\[
\widehat{h}= h+O(1).
\]
\end{prop}
%%=========================証明======================
\begin{proof}
We shall give a proof of the assertion for $\lambda$ satisfying $0<|\lambda|<1$.
If $1<|\lambda|$,
we can prove it similarly using the inverse map of $R$ instead of $R$.
See also \cite[Theorem 13]{KS3}.

First, we shall define $\widehat{h}$.
Let
\[
E := h\circ R-\Lambda^{-1}h
\]
be the error function.
There is a constant $C_0>0$ satisfying $\| E(x)\| <C_0$ for any $x\in S$.
We define
\[
\widehat{h}:=h+\sum_{n=0}^{\infty} \Lambda^{n+1}(E\circ R^n).
\]
To prove that it is well-defined and satisfies the normalization condition,
it suffices to prove that the series
\[
\sum_{n=0}^{\infty} ||\Lambda^{n+1}(E\circ R^n(x))||
\]
converges and
is bounded by a constant which is independent of $x\in S$.
These assertions follow from the following calculation.
\begin{align}
&\phantom{\leq} \sum_{n=0}^{\infty} ||\Lambda^{n+1}(E\circ R^n(x))||\\
&\leq \sum_{n=0}^{\infty} (l+1)\cdot||\Lambda^{n+1}||\cdot ||E\circ R^n(x)||\\
&= \sum_{n=0}^{l-1} (l+1)\cdot||\Lambda^{n+1}||\cdot ||E\circ R^n(x)||
+\sum_{n=l}^{\infty} (l+1)\cdot||\Lambda^{n+1}||\cdot ||E\circ R^n(x)||\\
&\leq C_1 + C_0\cdot \sum_{n=l}^{\infty} (l+1)\cdot n^l \cdot |\lambda|^{n-l}
\text{\hspace{30mm}from Lemma \ref{Lemma: matrix} (b)}\\
&\leq C_2,
\end{align}
where $C_1$ and $C_2$ are constants independent of $x\in S$.

Next, we shall prove that $\widehat{h}$ satisfies the functional equation.
It follows from the following formal calculation.
\begin{align}
\widehat{h}\circ R
&= h\circ R + \sum_{n=0}^{\infty}\Lambda^{n+1}(E\circ R^{n+1})\\
&= h\circ R - E +\sum_{n=0}^{\infty}\Lambda^{n} E\circ R^{n}\\
&= \Lambda^{-1} h+\Lambda^{-1} \sum_{n=0}^{\infty}\Lambda^{n+1} E\circ R^{n}\\
&= \Lambda^{-1} \widehat{h}.
\end{align}

Finally we shall prove the uniqueness of $\widehat{h}$.
Let both $\widehat{h}$ and $\widehat{h}'$ be vector valued functions
satisfying the functional equation and the normalization condition.
We set $g:=\widehat{h}-\widehat{h}'$.
Then $g$ is a bounded function satisfying
the functional equation $g\circ R=\Lambda^{-1} g$.
Assume that $g(x)\neq 0$ for some $x\in S$.
Then we get
\begin{equation}
1\geq \lim_{n\to\infty}||g\circ R^n(x)||^{1/n}=||\Lambda^{-n}g(x)||^{1/n}=|\lambda|^{-1},
\end{equation}
where the first inequality follows from the boundedness of $g$,
and the last equality follows from Lemma \ref{Lemma: matrix} (e).
This contradicts the assumption $0<|\lambda|<1$.
Consequently, we have $g(x)=0$ for any $x\in S$.
\end{proof}

%%=================標準高さの主定理の証明===================
\begin{proof}[{Proof of Theorem \ref{Theorem: canonical heights}}]
If $0<|\lambda|<1$ (resp.\ $1<|\lambda|$), the assertion follows
by applying Proposition \ref{Proposition: axiomatic canonical heights}
to the set $\plim[f] X(\var{\Q})$, the vector valued height function
\[
h_D:={}^t(h_{D_0}, h_{D_1}, \ldots ,h_{D_l}) \circ \p_0,\]
and the right shift operator
\[
R\colon\plim[f] X(\var{\Q})\longrightarrow \plim[f] X(\var{\Q})
\]
(resp.\ left shift operator).
\end{proof}

%%====================小さい固有値の高さの有界性======================
\begin{prop}\label{Proposition: boundedness}
Let notation be the same as in Theorem \ref{Theorem: canonical heights}.
Moreover, assume that $0< |\lambda | <1$.
Then for every $j\ (0\leq j \leq l)$ and for every point $P\in X(\var{\Q})$,
the sequence $\{ h_{D_j}(f^n(P)) \}_{n\geq 0}$
is bounded.
\end{prop}
\begin{proof}
For a point $P\in X(\var{\Q})$, take an $f$-orbit $(P_n)_{n\in\Z}$ of $P$.
Let
\[
h_D:={}^t(h_{D_0},h_{D_1},\ldots ,h_{D_l})
\]
be the vector valued height function.
Take the canonical height function $\widehat{h}_{D_j}$
as in Theorem \ref{Theorem: canonical heights}
and let
\[
\widehat{h}_D:={}^t(\widehat{h}_{D_0},\widehat{h}_{D_1},\ldots ,\widehat{h}_{D_l})
\]
be the vector valued canonical height function.
There is a real number $C\in \R_{>0}$
such that for every $f$-orbit $(P_n)_{n\in\Z}$ and every $m\geq 0$,
we have
\begin{align}
\hphantom{=} ||h_D(f^m(P))||
&= ||h_D\circ\p_0(f^m((P_n)))||\\
&\leq ||\widehat{h}_D(f^m((P_n)))||+C\\
&= ||\Lambda^{m}\widehat{h}_D((P_n))||+C\\
&\leq (l+1)\cdot ||\Lambda^{m}||\cdot||\widehat{h}_D((P_n))||+C.
\end{align}
When $m$ goes to $\infty$,
the last term converges to $C$ by Lemma \ref{Lemma: matrix} (b).
Hence the assertion follows.
\end{proof}
%%=======================本当は標準高さは必要ないよ====================
\begin{rmk}
It is possible to prove Proposition \ref{Proposition: boundedness}
directly without using Theorem \ref{Theorem: canonical heights}.
But the canonical heights for Jordan blocks are interesting themselves,
and they make the proof clearer.
\end{rmk}

%%%%%%%%%%%%%%%============主定理の証明の§==========================
\section{Proof of Theorem \ref{Theorem: preperiodicity}}
\label{Section: On the main theorem}

%%============================主定理===========================
In this section, we shall prove Theorem \ref{Theorem: preperiodicity}.

Before giving the proof, we give easy lemmata in linear algebra
which we frequently use in the proof.

%%=======================固有値集合の包含関係の補題================
\begin{lem}
Let $U,V$ be finite dimensional vector spaces over a field, and
$\vphi_U\colon U\longrightarrow U$, $\vphi_V\colon V\longrightarrow V$ be linear self-maps on $U, V$, respectively.
\begin{parts}
\Part{(a)}If there is an injection $\iota\colon V\longrightarrow U$
satisfying $\iota\circ \vphi_V=\vphi_U\circ \iota$,
then we have $\EV(\vphi_V)\subset\EV(\vphi_U)$.
\Part{(b)}Let $\pi\colon V\longrightarrow U$ be a surjection
satisfying $\pi\circ \vphi_V=\vphi_U\circ \pi$.
Then $\EV(\vphi_U)$ coincides with the set
\[
\left\{\lambda\in\EV(\vphi_V)\ \middle|\ 
\exists v\in V \text{ s.t. }
\pi v\neq 0 \text{ and }
\vphi_V v=\lambda v
\right\}.
\]
\Part{(c)}Let notation be as in {\rm (b)}.
Let $\lambda\in \EV(\vphi_V)\backslash \EV(\vphi_U)$,
and let $v_0, v_1, \ldots , v_r\in V$ satisfy
\[
\vphi_V v_j = v_{j-1}+\lambda v_j \quad (0\leq j\leq r),
\]
where we set $v_{-1}:=0$.
Then we have $\pi v_j=0$ for every $j\ (0\leq j \leq r)$.
\end{parts}
\end{lem}
\begin{proof}
\begin{parts}
\Part{(a),(b)} The assertions are obvious.
\Part{(c)} The equality $\pi v_0=0$ follows from (b).
It is easy to see $\pi v_j=0$ by induction.
\end{parts}
\end{proof}
%%====================Picが多項式で消えちゃうよ==================
\begin{lem}\label{Lemma: annihilate}
There is a monic integral polynomial $P_f(t)\in\Z[t]$
such that $P_f(f^\ast)$ annihilates $\Pic(X)$.
\end{lem}
\begin{proof}
See \cite[Lemma 19]{KS3}.
\end{proof}
We also recall an important lemma about the finiteness of
the dimension of the $\Q$-vector space $V_H$.
%%===================Hで生成される空間は有限次元===================
\begin{lem}\label{Lemma: finiteness}
The $\Q$-vector subspace $V_H$ of $\Pic(X)_\Q$ is finite dimensional.
\end{lem}
\begin{proof}
Let $P_f(t)\in \Z[t]$ be the monic polynomial as in Lemma \ref{Lemma: annihilate}.
Consider the $\Q$-subspace $V_H'$ of $\Pic(X)_\Q$ spanned
by the set
\[
\{(f^n)^\ast H\ |\ 0\leq n\leq \deg P_f(t)-1\}.
\]
Since $P_f(f^\ast)$ annihilates $H$,
the space $V_H'$ satisfies $f^\ast(V'_H)\subset V'_H$.
Hence, we have $V_H=V_H'$, and so $V_H$ is finite dimensional.
See also the top of the proof of \cite[Theorem 3]{KS3}.
\end{proof}

%%=========================主定理の証明===========================
\begin{proof}[Proof of Theorem \ref{Theorem: preperiodicity}]
Set $(V_H)_\C:=V_H\otimes_\Q \C$.
We decompose the $\C$-vector space $(V_H)_\C$ to the Jordan blocks as
\[
(V_H)_\C =\bigoplus_{i=1}^{\nu}V_i.
\]
Here, each $V_i$ satisfies $f^\ast(V_i)\subset V_i$ and $f^\ast|_{V_i}$ is represented by a Jordan block matrix of eigenvalue $\lambda_i$ as in Definition \ref{dfn: Jordan block}.
By relabeling, we may assume that
\begin{eqnarray}
\left\{
\begin{array}{l}
\lambda_i\in \EV(f^\ast; \Nu(X)_\Q) \text{ with }
1<|\lambda_i| \text{ for }1\leq i \leq \sigma, \\
\lambda_i\in \EV(f^\ast; \Nu(X)_\Q) \text{ with }
0<|\lambda_i|<1 \text{ for } \sigma +1 \leq i \leq \tau, \text{ and}\\
\lambda_i \in \EV(f^\ast;\Pic(X)_\Q)\backslash \EV(f^\ast;\Nu(X)_\Q)
\text{ for } \tau+1 \leq i\leq \nu.
\end{array}\right.
\end{eqnarray}
By assumption, no $\lambda\in \EV(f^\ast;\Nu(X)_\Q)$
satisfies $|\lambda|=1$.
Let $\{D_{i,j}\ |\ 0\leq j \leq l_i\}$ be the $\C$-basis of $V_i$ satisfying
the following linear equivalences
\[
f^\ast D_{i,j} \sim D_{i,j-1}+\lambda_i D_{i,j} \quad (0\leq j \leq l_j),
\]
where we set $D_{i,-1}:=0$.
Take the canonical height for Jordan blocks as in Theorem \ref{Theorem: canonical heights}
for each $1\leq i \leq \sigma$ and $0\leq j \leq l_i$.

If a point $P\in X(\var{\Q})$ is preperiodic point under $f$
we have
\[
\alpha_f(P)=1<\mu_H(f).
\]

Conversely, we assume that $\alpha_f(P)<\mu_H(f)$.
We shall prove that $P$ is preperiodic under $f$.
If $\widehat{h}_{D_{i,j}}(P)\neq 0$ for some index $(i,j)$ with $1\leq i\leq \sigma$,
fix such an index $i_0$ and let $j_0$ be the smallest index
satisfying $\widehat{h}_{D_{i_0,j_0}}(P)\neq 0$.
Then we have
\begin{align}
\widehat{h}_{D_{i_0,j_0}}(f^n(P))
&= \sum_{j=0}^{j_0}\binom{n}{j}\lambda_{i_0}^{n-j}\widehat{h}_{D_{i_0,j}}(P)\\
&= \lambda_{i_0}^n\widehat{h}_{D_{i_0,j_0}}(P).
\end{align}
Consequently, the arithmetic degree is bounded as follows.
\begin{align}
\alpha_f(P)
&= \lim_{n\to\infty} \max\{h_H(f^n(P)),1\}^{1/n}\\
&\geq \lim_{n\to\infty} |h_{D_{i_0,j_0}}(f^n(P))|^{1/n}\\
&\geq \lim_{n\to\infty} \left(|\widehat{h}_{D_{i_0,j_0}}(f^n(P))|-C\right)^{1/n}\\
&= \lim_{n\to\infty} \left(|\lambda_{i_0}^{n}\widehat{h}_{D_{i_0,j_0}}(P)|-C\right)^{1/n}\\
&= |\lambda_{i_0}|.
\end{align}
But since $\lambda_{i_0}\in \EV(f^\ast; \Nu(X)_\R)$,
we have $\alpha_f(P)\geq |\lambda_{i_0}| \geq\mu_H(f)$.
This is a contradiction.
Thus, we now have $\widehat{h}_{D_{i,j}}(P)=0$
for every $1\leq i \leq \sigma$ and $0\leq j \leq l_i$.

Write $H=\sum_{i,j}a_{i,j}D_{i.j}$ and fix an ample height $h_H$ with $h_H\geq 1$.
Since for $\tau+1 \leq i\leq \nu$,
the $\C$-divisors $D_{i,j}$ are algebraically equivalent to $0$,
the following inequality holds on $X(\var{\Q})$:
\begin{equation}
\left|h_H-\sum_{i=1}^{\tau}\sum_{j=0}^{l_i}a_{i,j}h_{D_{i,j}}\right|\leq o(h_H)
\end{equation}
(see \cite[Theorem B.3.2 (f)]{HS}).
Now, from the fact we proved above and Proposition \ref{Proposition: boundedness},
for $1\leq i \leq \tau$ and $0\leq j \leq l_i$,
the heights $h_{D_{i,j}}$ are uniformly bounded by a constant
on the forward $f$-orbit of $P$.
Consequently, we can find a constant $C>0$
such that the following inequalities
\begin{align}
h_H(f^n(P))-C
&\leq \left|h_H(f^n(P))-\sum_{i=1}^{\tau}\sum_{j=0}^{l_i}a_{i,j}h_{D_{i,j}}(f^n(P))\right|\\
&\leq o(h_H(f^n(P)))
\end{align}
hold for all $n\geq 0$.
The finiteness of the number of elements of the set
$\{h_H(f^n(P))\ |\ n\geq 0\}$ follows from this inequality.
Since the ample height function satisfies Northcott's property,
the point $P$ is preperiodic under $f$.
\end{proof}

%%============================応用の§===========================
\section{Applications to arithmetic and dynamical degrees}
\label{Section: Proof of the application to KSC}

In this section, we provide an application of Theorem \ref{Theorem: preperiodicity}
to the conjecture stated by Kawaguchi and Silverman in \cite[Conjecture 6]{KS2}.
Let notation be the same as in Section \ref{Section: Notation and definitions}.
%%======================力学系次数について======================
%\begin{dfn}
%The ({\it first}) {\it dynamical degree} of a surjective endomorphism
%$f\colon X \longrightarrow X$
%is defined by
%\[
%\delta_f:= \max_{\lambda\in \EV(f^\ast; \Nu(X)_\Q)}|\lambda|.
%\]
%\end{dfn}
%=======================力学系次数について=======================
\begin{dfn}
The {\it (first) dynamical degree} of a surjective endomorphism
$f\colon X\longrightarrow X$
is defined by
\begin{equation}\label{eqn: def of algebraic dynamical degree}
\delta_f := \lim_{n\to \infty} ((f^n)^\ast H \cdot H^{\dim X-1})^{1/n}.
\end{equation}
\end{dfn}
\begin{rmk}\label{rmk: dynamical degree 1}
One can see that
\[
\delta_f= \max_{\lambda\in \EV(f^\ast;\Nu(X)_\Q)}|\lambda|
\]
by the same way as Remark \ref{rmk: dynamical degree 3}.
More generally, the dynamical degree is defined for dominant rational self-maps.
The existence of the limit defining $\delta_f$
and the independence of $\delta _f$ on the choice of $H$ are known.
It is also known that $\delta_f$ is equal to
\[
\lim_{n\to \infty} \max_{\lambda^{(n)}\in \EV((f^n)^\ast;\ \Nu(X)_\Q)}|\lambda^{(n)}|^{1/n};
\]
see \cite[Proposition 1.5]{Gue}.
\end{rmk}

%============================application==============================
Recall that $V_H$ is a $\Q$-linear subspace of $\Pic(X)_\Q$ spanned by
the set $\{(f^n)^\ast H\ |\ n\geq 0\}$,
and $\var{V_H}$ is the image of $V_H$ in $\Nu(X)_\Q$.

\begin{lem}
The map $f_\ast \colon \Pic(X)_\Q\longrightarrow \Pic(X)_\Q$ induces
the endomorphism on $V_H$.
\end{lem}
\begin{proof}
Since $f_\ast f^\ast$ is the multiplication by $\deg f$ on $\Pic(X)$,
the map $f^\ast$ is injective on $\Pic(X)_\Q$.
In particular, the map $f^\ast|_{V_H}\colon V_H \longrightarrow V_H$ is injective.
Because $V_H$ is a finite dimensional $\Q$-vector space,
the map $f^\ast |_{V_H}$ is surjective.
For a $\Q$-divisor $D \in V_H$, there is a $\Q$-divisor $D'\in V_H$
such that $f^\ast D'=D$ in $\Pic(X)_\Q$.
Hence we get
\[
f_\ast D = f_\ast f^\ast D'= \deg f\cdot D' \in V_H.
\]
\end{proof}

\begin{prop}\label{Theorem: application}
Let $X$ be a smooth projective variety
defined over $\var{\Q}$.
Let $H$ be an ample divisor on $X$ defined over $\var{\Q}$, and
$f\colon X \longrightarrow X$ a surjective endomorphism on $X$
over $\var{\Q}$.
Assume that $\delta_f > (\deg f)^2$ and $\dim_\Q \overline{V_H}\leq 2$.
Then for every point $P\in X(\var{\Q})$, $P$ is wandering under $f$
(i.e., the forward $f$-orbit of $P$ is an infinite set)
if and only if $\alpha_f(P)=\delta_f$.
In particular, if the forward $f$-orbit of $P$ is Zariski dense in $X$,
we have $\alpha_f(P)=\delta_f$.
\end{prop}
\begin{proof}
By Remark \ref{rmk: dynamical degree 3},
the dynamical degree $\delta_f$
appears as an eigenvalue of
$f^\ast \colon \var{V_H}\longrightarrow \var{V_H}.$
If we have $\dim_\Q \overline{V_H}=1$, the ample $\R$-divisor $H$
satisfies the following numerical equivalence
\[
f^\ast H \equiv \delta_f H.
\]
In this case, it is well-known that $P$ is wandering under $f$ if and only if $\alpha_f(P)=\delta_f$.
See \cite[Proposition 7]{KS1} for details.

So we may assume $\dim_\Q\overline{V_H}=2$.
Since $f^\ast\colon \var{V_H}\longrightarrow \var{V_H}$
and $f_\ast\colon \var{V_H} \longrightarrow \var{V_H}$
come from the $\Z$-linear self-maps on $\Nu(X)$,
their eigenvalues are algebraic integers.
Let $\det (f^\ast |_{\var{V_H}})$ and $\det (f_\ast |_{\var{V_H}})$
be the determinants of the $\Q$-linear maps
$f^\ast\colon \var{V_H} \longrightarrow \var{V_H}$ and
$f_\ast\colon \var{V_H}\longrightarrow  \var{V_H}$, respectively.
Since $\det(f_\ast |_{\var{V_H}})$ is a non-zero rational number
which is also an algebraic integer, we have
\[
\det(f_\ast|_{\var{V_H}})\geq 1.
\]
Thus, we have
\begin{align}\label{inequality: degrees}
\det (f^\ast |_{\var{V_H}}) &\leq \det (f_\ast |_{\var{V_H}})\cdot\det (f^\ast|_{\var{V_H}})\\
&= \det (\deg f\cdot\id_{\var{V_H}})\\
&= (\deg f)^2.
\end{align}

Let $\{\delta_f, \lambda\}$ be the eigenvalues of
$f^\ast\colon \var{V_H}\longrightarrow \var{V_H}$.
Since we have
\[
\delta_f \cdot \lambda= \det (f^\ast |_{\var{V_H}}),\]
we obtain
\begin{equation}
\lambda=\frac{\det (f^\ast |_{\var{V_H}})}{\delta_f}\leq
\frac{(\deg f)^2}{\delta_f}< 1
\end{equation}
by the assumption $\delta_f> (\deg f)^2$.
Hence, we conclude that $\mu_H(f)=\delta_f>1$.

Now let $P\in X(\var{\Q})$ be a wandering point under $f$.
By Theorem \ref{Theorem: preperiodicity},
we get $\alpha_f(P)\geq \mu_H(f)=\delta_f$.
The opposite inequality $\alpha_f(P)\leq \delta_f$ is known in general;
see \cite[Theorem 4]{KS2}, \cite[Theorem 1.4]{Mat} for dominant rational maps, and see also the last sentence of the proof of Theorem 3 in \cite{KS3} for surjective endomorphisms.
Hence the assertion $\alpha_f(P)=\delta_f$ follows.
\end{proof}

\begin{rmk}
If $f\colon X\longrightarrow X$ does not satisfy the assumption $\delta_f>(\deg f)^2$,
the assertion that $P$ is wandering under $f$ if and only if $\alpha_f(P)=\delta_f$
may not be true.
For example, consider the elliptic curves $E$ and $E'$ over $\var{\Q}$,
and a non-torsion point $P_0\in E'(\var{\Q})$.
Then for the self-morphism $f\colon E\times E' \longrightarrow E\times E'$
defined by
\[
f(Q,R)=([2]Q, R+P_0),
\]
we have
\[
\deg f=\delta_f=4.
\]
Every rational point $(Q,R)\in E(\var{\Q})\times E'(\var{\Q})$
is wandering under $f$.
On the other hand, the arithmetic degree $\alpha_f(Q,R)$
is equal to $4$ if and only if $Q$ is non-torsion.
\end{rmk}

\begin{exam}
The following example was studied by Silverman; see \cite[Theorem 1.1]{Sil} for details.
Let $X$ be the smooth projective surface contained in $\P^2\times \P^2$
defined over $\var{\Q}$
which is given by the intersection of an effective divisor
of type $(1,1)$ and an effective divisor of type $(2,2)$.
Thus $X$ is the locus of the points
\[
((x_1:x_2:x_3),(y_1:y_2:y_3))\in \P^2\times \P^2
\]
satisfying the following equations
\[
\sum_{i,j=1}^3 a_{i,j}x_i y_j=\sum_{i,j,s,t=1}^3 b_{i,j,s,t}x_i x_j y_s y_t =0.
\]
The projections $p_i\colon X\longrightarrow \P^2\ (i=1,2)$
induced by the natural projections
\[
p_i\colon \P^2\times \P^2\longrightarrow \P^2\quad (i=1,2)
\]
are both double covers of $\P^2$.
For $i=1,2$, let $\sigma_i\in\Aut(X)$ be the involution induced by the double cover $p_i$.
Then the automorphism $f:=\sigma_1\circ \sigma_2$ on $X$
satisfies the condition of Proposition \ref{Theorem: application}.
For $i=1,2$, let $D_i$ be the divisor on $X$ which is given by pulling back
the hyperplane section of $\P^2$ by $p_i$.
Then the sum $H:= D_1+D_2$ is an ample divisor on $X$, and the $\Q$-linear subspace $\var{V_H}$ is spanned by $D_1$ and $D_2$; see \cite[Lemma 2.1]{Sil} and \cite[Proposition 2.5]{Sil}.

%Let $E^+:= \alpha D_1-D_2$, and $E^-:= -D_1+\alpha D_2$,
%where we put $\alpha:=2+\sqrt{3}$.
%Then the sum $H:=E^+ + E^-$ is an ample divisor on $X$; see \cite[Proposition 2.5]{Sil}.
%Furthermore, $E^+$ and $E^-$ satisfy the linear equivalences
%$f^\ast E^+ \sim \alpha^{2} E^+$ and $f^\ast E^- \sim \alpha^{-2} E^-$,
%respectively; see \cite[Lemma 2.1]{Sil}.
%Hence we have $\dim \var{V_H}=2$.
\end{exam}

\begin{exam}
Let $X:= E\times E$, where $E$ is an elliptic curve defined over $\var{\Q}$
without complex multiplication.
Let $f\colon X\longrightarrow X$ be the endomorphism defined by
\[
(x,y)\mapsto (y,-x+[4]y).
\]
Let $D_i:=p_i^\ast H_E$, where $H_E$ is an ample divisor on $E$,
and $p_i\colon X\longrightarrow E\ (i=1,2)$ is the projection to the $i$-th component.
Let $\Delta$ be the diagonal divisor in $X$, and we put $D_3:=D_1+D_2-\Delta$.
Then we have
\[
\Nu(X)_\R=\R D_1\oplus \R D_2 \oplus \R D_3.\]
We calculate $f^\ast D_i\ (i=1,2,3)$ as follows:
\begin{align}
f^\ast D_1 &\equiv D_2\\
f^\ast D_2 &\equiv D_1 +16 D_2 -4 D_3\\
f^\ast D_3 &\equiv 8D_2 -D_3.
\end{align}
%Hence the representation matrix $A$ of $f^\ast$ associated with $\{D_1,D_2,D_3\}$ is
%\[A=\left(
%\begin{array}{ccc}
%0 & 1 & 0\\
%1 &16 & -4\\
%0 & 8 & -1
%\end{array}
%\right).
%\]
Then the following $\R$-divisors
\begin{align}
D_1' &= D_1 +D_2+4D_3,\\
D_2' &= D_1 +(7+4\sqrt{3})D_2+(-2-\sqrt{3})D_3,\ \text{and}\\
D_3' &= D_1 +(7-4\sqrt{3})D_2+(-2+\sqrt{3})D_3
\end{align}
satisfy the following numerical equivalences
\begin{align}
f^\ast D_1' &\equiv D_1',\\
f^\ast D_2' &\equiv (7+4\sqrt{3})D_2',\ \text{and}\\
f^\ast D_3' &\equiv (7-4\sqrt{3})D_3'.
\end{align}
Since the nef cone in $\Nu(X)_\R$ is
\[
\{aD_1+bD_2+cD_3\ |\ ab-c^2\geq 0,\ a+b+c \geq 0\},
\]
the divisor
\[
H:=D_2'+D_3'=2D_1+14D_2-4D_3
\]
is an ample divisor; see \cite[Section 2, (2.0.1)]{BS}.
By construction, we have
\[
\dim_\Q V_H=\dim_\R(V_H)_\R=2
\]
and
\[
\mu_H(f)=\delta_f= 7+4\sqrt{3}.
\]
By Theorem \ref{Theorem: preperiodicity},
a point $P\in X(\var{\Q})$ is preperiodic under $f$
if and only if we have $\alpha_f(P)=1$.
Now $f$ is an automorphism whose inverse map is given by
\[
(x,y)\mapsto ([4]x-y,x).
\]
On the other hand,
\[
g:=[2]\circ f\colon X\longrightarrow X
\]
is a surjective endomorphism, which is not an automorphism.
Since we have
$g^\ast = 4\cdot f^\ast$ on $\Nu (X)$,
the eigenvalues of $g^\ast|_{\var{V_H}}$
are $28+16\sqrt{3}\ (>1)$ and $28-16\sqrt{3}\ (<1)$.
Hence a point $P\in X(\var{\Q})$ is preperiodic under $g$
if and only if we have $\alpha_g(P)<28 + 16 \sqrt{3}$.
\end{exam}

We provide some remarks on the dynamical degrees and
the difference between the Jordan blocks of
$f^\ast \colon V_H\longrightarrow V_H$ and $f^\ast\colon \overline{V_H}\longrightarrow \overline{V_H}$.
%=========================力学系次数のRemark 1===================
\begin{rmk}\label{rmk: dynamical degree 1}
Let $X$ be a smooth projective complex variety and
$f\colon X\dashrightarrow X$ a dominant rational map.
Fix an closed immersion $\iota\colon X\longrightarrow \P^N$
and a hyperplane $H$ in $\P^N$.
Put $H_X:= \iota^\ast H$ and $\omega_X:=\omega_{FS}$,
where $\omega_{FS}$ is the Fubini-Study form on $\P^N$.
Then we have
\begin{align}
\delta_f
&= \lim_{n\to\infty} ((f^n)^\ast H_X \cdot H_X^{\dim X-1})^{1/n}\\
&= \lim_{n\to \infty} (\int (f^n)^\ast \omega_X\wedge \omega_X^{\dim X -1})^{1/n}\\
&= \lim_{n\to\infty} \| (f^n)^\ast\|_{1,1}^{1/n},
\end{align}
where $\|(f^n)^\ast \|_{1,1}$ is the operator norm of
$f^\ast\colon H^{1,1}(X,\C) \longrightarrow H^{1,1}(X,\C)$.
The second equality follows from the property
\cite[Corollary 19.2 (b)]{Ful}
of the cycle map
$\cl\colon \Nu(X)\longrightarrow H^{2}(X, \C)$
and the comparison theorem $H^2_{dR}(X)\cong H^2(X,\C).$
The third equality follows from \cite[Corollary 7]{DS} or \cite[Proposition 1.2 (iii)]{Gue}.
\end{rmk}
%===================力学系次数のremark 2================
\begin{rmk}\label{rmk: dynamical degree 2}
Let $X$ be a smooth projective variety over $\C$, and $f\colon X \longrightarrow X$
a surjective endomorphism.
Let $\omega \in \Pic(X)$ be an ample divisor class.
Let $\xi \in H^1(X,\mathcal{O}_X)=H^{0,1}(X)$ be the eigenvector of
$f^\ast\colon H^1(X,\mathcal{O}_X)\longrightarrow H^1(X,\mathcal{O}_X)$
associated with the eigenvalue whose complex
absolute value is the spectral radius $\lambda$
of $f^\ast$ on $H^{0,1}(X)$.
Then we have
\[
(\xi \cdot \overline{\xi}\cdot \omega^{2\dim X-2})>0
\]
by Hodge theory.
In particular $\xi\cdot \overline{\xi}\in H^{1,1}(X)$ is non-zero.
Hence the class $\xi \cdot \overline{\xi}$ is an eigenvector
of $f^\ast \colon H^{1,1}(X)\longrightarrow H^{1,1}(X)$
associated with the eigenvalue $|\lambda|^2$.
Consequently, the spectral radius of
$f^\ast$ on $H^{1,1}(X)$
is greater than or equal to the square of the spectral radius of
$f^\ast$ on $H^{0,1}(X)$.
%Then we have
%\begin{align}\label{eqn: Pic02<delta}
%(\lim_{n\to \infty} \|(f^n)^\ast\|_{0,1}^{1/n})^2
%\leq \lim_{n\to\infty} \|(f^n)^\ast\|_{1,1}^{1/n}
%\end{align}
%by \cite[Proposition 5.8]{Din},
%where $\|(f^n)^\ast\|_{0,1}$ is the operator norm of
%$(f^n)^\ast\colon H^{0,1}(X,\C)\longrightarrow H^{0,1}(X,\C)$.
\end{rmk}

%===================力学系次数のremark 3=================
\begin{rmk}\label{rmk: dynamical degree 3}
Assume that $f\colon X\longrightarrow X$ is a surjective endomorphism.
Then we can see
\[
\delta_f= \max_{\lambda\in \EV(f^\ast;\ \var{V_{H}})}|\lambda|
\]
by the following argument.
For $D\in \Nu(X)_\R$, we put
\begin{equation}
\|D \| :=\inf\left\{ (D_1\cdot H^{\dim X-1})+(D_2 \cdot H^{\dim X-1})\ \middle|\ 
\substack{D=D_1-D_2,\\ D_1, D_2\in \Nu(X) \text{ are effective}}\right\}.
\end{equation}
Then since $\|\cdot \|$ is a non-trivial norm on $\Nu(X)_\R$, the quantity
\[
\|(f^n )^\ast \|:=\sup_{D\in \Nu(X)_\R} \left( \frac{\| (f^n)^\ast D \|}{\| D\|} \right)
\]
is the operator norm of $(f^n)^\ast \colon \Nu(X)_\R\longrightarrow \Nu(X)_\R$.
Thus we see that
\begin{align}
\lim_{n\to\infty} ((f^n)^\ast H\cdot H^{\dim X-1})^{1/n}
&= \lim_{n\to\infty} \left( \frac{((f^n)^\ast H\cdot H^{\dim X-1})}{\| H\|} \right)^{1/n}\\
&\leq \lim_{n\to\infty} \left( \frac{\| (f^n)^\ast H \|}{\| H\|} \right)^{1/n}\\
&\leq \lim_{n\to\infty} \sup_{D\in \Nu(X)_\R} \left( \frac{\| (f^n)^\ast D \|}{\| D\|} \right)^{1/n}\\
&= \limsup_{n\to\infty} \| (f^n)^\ast \|^{1/n}\\
&= \max_{\lambda\in\EV(f^\ast;\ \Nu(X)_\R)}|\lambda|.
\end{align}
Hence we obtain the following inequalities
\begin{align}
\delta_f &= \max_{\lambda\in\EV(f^\ast;\ \Nu(X)_\R)}|\lambda|\\
&\geq \max_{\lambda\in\EV(f^\ast;\ (\var{V_{H}})_\R)}|\lambda|\\
&\geq \lim_{n\to\infty} ((f^n)^\ast H\cdot (H)^{\dim X-1})^{1/n}\\
&= \delta_f.
\end{align}
Hence $\delta_f$ is equal to the spectral radius of
$f^\ast \colon \var{V_H}\longrightarrow \var{V_H}$.
Since $f^\ast$ preserves the cone generated by ample classes in $\var{V_H}$,
we also have $\delta_f \in \EV (f^\ast ;\ (\var{V_H})_\R)$ by \cite[Theorem]{Bir}.
\end{rmk}

\begin{rmk}
Let notation be as in Section \ref{Section: Notation and definitions}.
It is a natural problem to understand
the difference of $\EV(f^\ast; V_H)$ and $\EV(f^\ast; \var{V_H})$,
and the difference of the size of the Jordan blocks
of the linear maps $f^\ast|_{V_H}$ and $f^\ast|_{\var{V_H}}$.

Let $L$ be the $\Z$-submodule of $\Pic(X)$ generated
by the set $\{(f^n)^\ast H\ |\ n\geq 0\}$.
Let $\var{L}$ be the image of $L$ in $\Nu(X)$.
By \cite[Lemma 19]{KS3},
$L$ and $\var{L}$ are finitely generated abelian groups.
If it is necessary, by replacing $H$ by a multiple of $H$,
we may assume $L$ and $\var{L}$ do not have torsion elements.
Let $\pi \colon L\longrightarrow \var{L}$ be the canonical surjection,
and put $M:=\ker \pi$.
We also put $W\colon= M\otimes _{\Z} \Q$.
Then $M$ is a subgroup of $\Pic^0(X)$.
Let
\[
\varphi \colon \Pic^0(X) \longrightarrow \Pic^0(X)
\]
be the morphism induced by $f$.
Let
\[
\vphi' \colon H^1(X,\mathcal{O}_X)\longrightarrow H^1(X,\mathcal{O}_X)
\]
be the $\C$-linear map induced by $f$.
Let $F(t)$ be the characteristic polynomial of $\vphi'$.
Since $H^1(X,\mathcal{O}_X)$ is isomorphic to the Lie algebra of $\Pic^0(X)$,
we have $F(\vphi)=0$ on $\Pic^0(X)$.
Hence in particular, we get $F(\vphi |_M)=0$.
Since $W$ is generated by $M$, we also get $F(\vphi |_M)$=0.
Consequently, the spectral radius of $f^\ast \colon W\longrightarrow W$
is less than or equal to the spectral radius of $\vphi'$.
Combining Remark \ref{rmk: dynamical degree 1}, Remark \ref{rmk: dynamical degree 2},
and Remark \ref{rmk: dynamical degree 3},
we obtain that the spectral radius of $f^\ast \colon W \longrightarrow W$ is less than or equal to $\sqrt{\delta_f}$.

Therefore, for $\lambda\in\C$ with $|\lambda|>\sqrt{\delta_f}$,
the Jordan normal form of $f^\ast \colon V_H\longrightarrow V_H$
associated with the eigenvalue $\lambda$
is identified with the Jordan normal form of
$f^\ast \colon \overline{V_H}\longrightarrow \overline{V_H}$
associated with $\lambda$ by the canonical projection
$\pr \colon V_H \longrightarrow \overline{V_H}$.
\end{rmk}

\section*{Acknowledgments}
The author is grateful for the support of the Top Global University project for Kyoto University (abbrev.\ the KTGU project).
His co-supervisor by the KTGU project, Professor Joseph H. Silverman at Brown University
kindly, supported the author's stay at Brown University by the KTGU project.
The author would like to thank his advisor Tetsushi Ito
for carefully reading early versions of this paper many times and
precisely pointing out inaccuracies. 
The author would also like to thank Yohsuke Matsuzawa, and Takahiro Shibata
for giving me some helpful advice.

\end{document}